\documentclass[11pt,letterpaper]{article}
\usepackage[utf8]{inputenc}
\usepackage{amsmath}
\usepackage{amsfonts}
\usepackage{mathrsfs}
\usepackage{amssymb}
\usepackage{amsthm}
\usepackage{graphicx}
\usepackage{hyperref}
\usepackage{enumerate}

\usepackage{soul}
\usepackage[dvipsnames]{xcolor}
\sethlcolor{yellow}

\usepackage{tikz}
\usetikzlibrary{arrows}
\usetikzlibrary{decorations.pathmorphing, positioning, arrows.meta}

\tikzstyle{sm-vert}=[fill=black, draw = black, shape = circle, inner sep = 2.5pt]
\tikzstyle{dual-vert}=[fill=white, draw = blue, very thick, shape = circle, inner sep = 2.5pt]

\tikzstyle{negative}=[-, draw=red, very thick, dashed]
\tikzstyle{positive}=[-, draw = black, very thick]
\tikzstyle{high-pos}=[-, draw = black, very thick, preaction={
draw,yellow,-,
double=yellow,
double distance=2\pgflinewidth,
}]
\tikzstyle{dual-edge}=[-, draw=blue, very thick]
\tikzstyle{high-dual}=[-, draw = blue, very thick, preaction={
draw,yellow,-,
double=yellow,
double distance=2\pgflinewidth,
}]
\tikzstyle{neg path}=[-, draw=red, draw=red, very thick,  dash pattern = on 0pt off 3pt, line cap = round]

\author{Nicholas Lacasse\footnote{Department of Mathematical Sciences, Binghamton University, State University of New York, Binghamton, NY 13902-6000. Email: lacasse@math.binghamton.edu} }
\title{Minimal and Disjoint Negation Sets in Signed Graphs}
\newtheorem{thm}{Theorem}
\newtheorem{prop}[thm]{Proposition}
\newtheorem{lemma}[thm]{Lemma}

\newtheorem{cor}[thm]{Corollary}

\begin{document}

\maketitle

\noindent
\textbf{Abstract.} A signed graph is a graph with a function that assigns a label of positive or negative to each edge. The sign of a circle is the product of the signs of its edges; a graph is balanced if all of its circles are positive. A set of edges whose negation yields a balanced graph is a negation set. Results: tests to determine whether a negation set is minimal, minimum, or the unique minimum; any two disjoint negation sets must be bipartite; two classes of graphs are shown to have bipartite negation sets (in general, existence is an unsolved problem); I give an algorithm which finds a maximum family of disjoint negation sets that includes a given negation set.



\section{Introduction}

A \emph{signed graph} $\Sigma$ is a pair $(\Gamma, \sigma)$ where $\Gamma$ is a graph and $\sigma$ is a function from $E(\Gamma)$ to $\lbrace +, - \rbrace$. The \emph{sign} of a path is the product of the signs of its edges. Particularly important are the signs of closed paths, which are called \emph{circles}. A signed graph is \emph{balanced} if all of its circles are positively signed. Balance is a fundamental property of signed graphs. A useful metric of proximity to balance is the \emph{frustration index}, the minimum number of edges whose negation results in a balanced graph. Any set of edges whose negation yields a balanced graph is called a \emph{negation set}. To \emph{switch} a signed graph by a subset of its vertices $X \subseteq V(\Sigma)$ means to negate every edge with exactly one end in $X$. The collection of graphs to which $\Sigma$ can be switched is called the \emph{switching class} of $\Sigma$. 

The graphs in this paper are simple and connected. A \emph{Harary bipartition} of a signed graph $\Sigma$ is a bipartition of its vertex set, $X \cup Y = V(\Sigma)$, such that the negative edges of $\Sigma$ are precisely the edges between $X$ and $Y$. A signed graph is balanced if and only if it has a Harary bipartition \cite{Harary}. A path is \emph{fully negative} if every edge in the path is negative, similarly for \emph{fully positive}. The maximum degree of a graph $G$ is denoted by $\Delta G$. The degree, positive degree, and negative degree of a vertex $v$ are denoted by $d(v)$, $d^{+}(v)$, and $d^{-}(v)$ respectively. The graph obtained from $\Sigma$ by switching a vertex $v$ is denoted by $\Sigma^v$. The graph obtained by switching a set of vertices $X \subseteq V(\Sigma)$ is denoted by $\Sigma^X$. A signed graph $\Sigma$ with a subset of its edges $Y \subseteq E(\Sigma)$ negated is denoted by $\Sigma_Y$. A signed graph is \emph{antibalanced} if it can be switched so every edge is negative, i.e., if $E(\Sigma)$ is a negation set of $\Sigma$. This follows from a well-known proposition.

\begin{prop}\label{negation-negative}
The negation sets of a signed graph are precisely the negative edge sets that can be obtained by switching the graph.
\end{prop}

For a graph $G$, the $k$-core of $G$ is the graph that is obtained by repeatedly deleting vertices of degree less than $k$. For $X \subseteq V(G)$, the cut determined by $X$ is denoted by $c(X, X^C)$ or just $c(X)$. The symmetric difference of two sets, $S$ and $T$, is denoted by $S \triangle T$. A set of vertices $X \subseteq V(G)$ is \emph{stable} if no edge has both its incident vertices in $X$. A \emph{stable bipartition} of a graph $G$ is a bipartition of $V(G)$ where both classes are stable. A graph has a stable bipartition if and only if it is bipartite. A stable bipartition of a graph is unique if and only if $G$ is connected. I use $[n]$ to denote the set $\lbrace i \rbrace_{i=1}^n$.

It may seem like negation sets would be well understood, especially in light of proposition \ref{negation-negative}. But many of their features are not known and do not offer obvious paths to understanding. Some features of negation sets, even once understood, are shown to be computationally difficult. I have worked to understand negation sets by studying two aspects of them: the minimal negation sets and collections of disjoint negation sets. 

In figures, positive edges are represented by solid black lines, negative edges by dashed red lines, and paths where all edges are negative are represented by red dotted lines.

\section{Minimal Negation Sets}

A negation set is \emph{minimal} if no proper subset of it is a negation set. It is \emph{minimum} if there is no negation set consisting of fewer edges. Perhaps the most fundamental question is: can we tell if a negation set is minimal? We can and the test is remarkably easy.

\begin{prop}\label{minimal-check}
Let $\Sigma$ be a signed graph with negation set $B$. The negation set $B$ is minimal if and only if $B$ does not contain a cut of $\Sigma$ (i.e., $\Sigma \setminus B$ is connected).
\end{prop}

\begin{proof}
Suppose $B$ is exactly the negative edge set of $\Sigma$. If $B$ contains a cut $C$, then $\Sigma_{B \setminus C}$ is balanced because its negative edges are exactly a cut. So $B \setminus C$ is a negation set. Hence $B$ is not minimal. 

Conversely, if $B$ is not minimal, then there is some $C \subseteq B$ such that $\Sigma_{B \setminus C}$ is balanced. But the negative edges of $\Sigma_{B \setminus C}$ are exactly $C$, so they must be a cut of $\Sigma$. 
\end{proof}

The simple characterization of minimality encourages us to find a test to determine whether a negation set is minimum. I have results which tell us whether a negation set is minimum (or even the unique minimum), but none of those is a total characterization.

\begin{lemma}\label{suff-lem-one}
Let $\Sigma$ be a signed graph with negation set $B$. If there exists a collection of $|B|$ edge-disjoint negative circles, then $B$ is a minimum negation set.
\end{lemma}

\begin{proof}
Since a negation set must change the sign of every negative circle, then any negation set must be at least size $|B|$ (here it is important the circles are disjoint). But $B$ has exactly size $|B|$. So $B$ is minimum. 
\end{proof}

\begin{prop}\label{chrom-prop}
Let $\Sigma = (K_n, \sigma)$ with negation set $B$. If $n - |V(B)| \geq \chi'(B)$, then $B$ is a minimum negation set.
\end{prop}

\begin{proof}
I will show the sufficient condition of lemma \ref{suff-lem-one} is satisfied by finding a negative triangle through each edge of $B$. Properly color the edges in $B$ with the integers from 1 to $\chi'(B)$. Choose vertices $v_1, \ldots, v_{\chi'(B)}$ from $V(\Sigma) - V(B)$. This is possible because $n - |V(B)| \geq \chi'(B)$. For any edge in $B$ with color $i$, its ends form a triangle with $v_i$. All triangles that can be constructed in this way form the collection of $|B|$ triangles. See, for example, figure \ref{edge-disjoint-triangles}.

Each triangle has exactly one edge in $B$ and is therefore negative. If two of the triangles have adjacent edges in $B$, then those two edges must have different colors. So the only vertex these triangles share in common is where their two negative edges meet and thus they cannot share any edges. Since the collection of negative triangles satisfies the hypotheses of lemma \ref{suff-lem-one}, $B$ is a minimum negation set.
\end{proof}

\begin{figure}
\caption{Edge-disjoint negative triangles found by edge coloring as described in the proof of proposition \ref{chrom-prop}.}
\label{edge-disjoint-triangles}
\begin{center}
\includegraphics[scale=1]{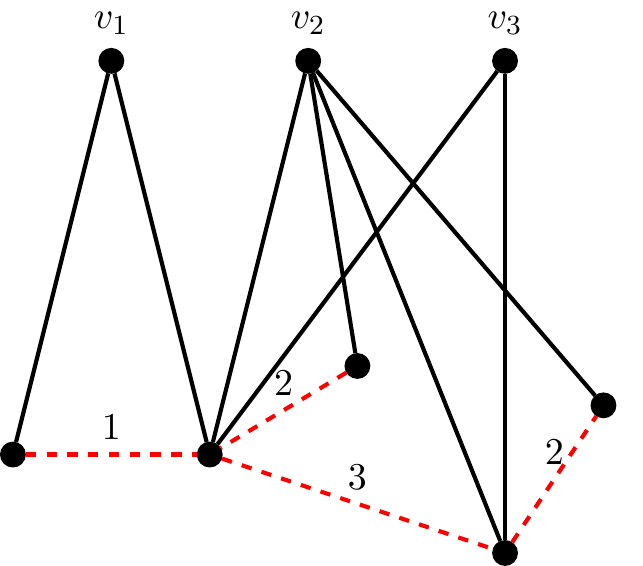}
\end{center}
\end{figure}

Propositions \ref{suff-two-neg-circles} and \ref{suff-small-B-unique} provide sufficient conditions for a negation set to be the unique minimum negation set. 

\begin{prop} \label{suff-two-neg-circles}
Let $\Sigma$ be a signed simple graph with minimum negation set $B$ as its negative edge set. If for each $e_i \in B$ there exist two negative circles $C_{i_1}, C_{i_2}$ both containing $e_i$ such that the edge sets in $\lbrace E(C_{i_1}) \cup E(C_{i_2}) \mid e_i \in B \rbrace$ are pairwise disjoint, then $B$ is the unique minimum negation set.
\end{prop}

\begin{proof}
Suppose the hypotheses and let $C$ be a negation set with $|C| = |B|$. For each $i$, $C$ contains an edge from $E(C_{i_1}) \cup E(C_{i_2})$. Since $E(C_{i_1}) \cup E(C_{i_2})$ is disjoint from $E(C_{j_1}) \cup E(C_{j_2})$ for $i \neq j$, then $C$ must contain exactly one edge from each $E(C_{i_1}) \cup E(C_{i_2})$. But since $E(C_{i_1}) \cup E(C_{i_2})$ corresponds to two negative circles sharing only a single edge, the only way to balance both of these negative circles is to choose the edge in $B$. Hence $C = B$.
\end{proof}

\begin{prop}\label{suff-small-B-unique}
Let $\Sigma = (K_n, \sigma)$ with negation set $B$. If $|B| \leq \frac{n}{2} - 1$ then $B$ is the unique minimum negation set of $\Sigma$.
\end{prop}

\begin{proof}
Suppose the negative edges of $\Sigma$ are exactly $B$. It suffices to show every switching of $\Sigma$ has more than $\frac{n}{2}-1$ negative edges. Let $X \subseteq V(\Sigma)$. The number of negative edges in $\Sigma^X$ is the number of negative edges in $\Sigma$ not in $c(X)$ plus the number of positive edges of $\Sigma$ in $c(X)$. Thus $\Sigma^X$ has $$|c(X)| - |c(X) \cap B| + |B \setminus c(X)|$$ negative edges. Because $|c(X)| \geq n-1, |c(X) \cap B| \leq \frac{n}{2}-1$, and $|B \setminus c(X)| \geq 0$, then $$|c(X)| - |c(X) \cap B| + |B \setminus c(X)| \geq (n-1) - (\frac{n}{2} -1) + 0 = \frac{n}{2}.$$ So any switching of $\Sigma$ will necessarily have more negative edges. So $B$ is the unique minimum.
\end{proof}

\section{Disjoint Negation Sets}
For the rest of the paper, I will use the \emph{positive and negative subgraphs} of $\Sigma$ which are obtained by deleting the negative and positive edges, respectively, from $\Sigma$.  They are denoted by $\Sigma^+$ and $\Sigma^-$.  I will also use the \emph{edge-induced positive subgraph} and the \emph{edge-induced negative subgraph} which are the signed graphs induced by $E^+$ and $E^-$.  They are denoted by $\Sigma{:} E^+$ and $\Sigma{:} E^-$ respectively.

For two vertices $x$ and $y$, the \emph{positive distance} from $x$ to $y$ is the distance from $x$ to $y$ in the positive subgraph. It is denoted by $d^+(x, y)$. Given two sets of vertices, $X, Y \subseteq V(\Sigma)$, define $$d^+(X, Y) := \min_{(x, y) \in X \times Y} d^+(x, y).$$

\begin{thm}\label{disjoint-negation-sets-are-bipartite}
A negation set of a signed graph $\Sigma$ is bipartite if and only if there exists another negation set disjoint from it.
\end{thm}

\begin{proof}
Suppose $\Sigma$ has bipartite negative edge set $A$. Since $A$ is bipartite, it is contained in a cut of the graph, say $c(X)$. Then the negative edges of $\Sigma^X$ are disjoint from $A$. 

Conversely, suppose $\Sigma$ has disjoint negation sets $A$ and $B$ where $A = E^-(\Sigma)$ and $B = E^-(\Sigma^S)$ for some $S \subseteq V(\Sigma)$. By way of contradiction, suppose $A$ has a circle $C$ of odd length. Then $C$ is negatively signed. Since $A$ and $B$ are disjoint, then $B$ contains no edge of $C$. So $C$ is positively signed in $\Sigma^S$. But switching preserves the sign of circles. So, I have reached a contradiction. Thus $A$ and $B$ must be bipartite.
\end{proof}

\section{Bipartite Negation Sets}

Theorem \ref{disjoint-negation-sets-are-bipartite} informs us that if we want to pack negation sets, we will want to be looking at bipartite negation sets. This provokes a natural question: how do we find bipartite negation sets? This question currently has no answer. What's more, the situation is worse! How do we know that a signed graph has a bipartite negation set? Determining this was proven to be NP-complete: see \cite[theorem 4.2]{bipartite-complexity} and \cite{4-2-thm-improvement}. In \cite{packing-homomorphism-connection}, it is shown that determining whether a signed graph has at least $i$ disjoint negation sets can be turned into a signed graph homomorphism problem. In particular, a signed graph has a bipartite negation set if and only if there is a homomorphism from it to a negative digon with a positive loop at each vertex. $\mathcal{C}$

Some progress has been made on determining classes of graphs with bipartite negation sets. It was addressed by Hage and Harju (in different terminology) in \cite{Hage-Harju-Bipartite}. In \cite{Hage-Harju-Acyclic} they characterized, by forbidden subgraphs, which signed complete graphs have an acyclic negation set. They took a similar approach in \cite{Hage-Harju-Bipartite} to determine which signed complete graphs have a bipartite negation set but they were unable to complete the proof. To my knowledge, their proof has not been finished by anyone else. I will prove that two other classes of graphs have bipartite negation sets and then explore how many negation sets can be packed in a graph. 

I would like to thank Reza Naserasr for providing the simpler proof of theorem \ref{antibal-planar} presented here. 
 
\begin{thm}\label{antibal-planar}
Every antibalanced planar graph has a bipartite negation set. 
\end{thm}

\begin{proof}
Let $\Sigma$ be an antibalanced planar graph with underlying graph $\Gamma$. Suppose $E^{-}(\Sigma) = E(\Sigma)$. Let $\gamma$ be a proper 4-coloring of $\Gamma$. Let $V_{i} := \gamma^{-1}(i)$. The negative edges in $\Sigma^{V_3 \cup V_4}$ are all contained in the cut $c(V_1 \cup V_3, V_2 \cup V_4)$. So the negative edge set of $\Sigma^{V_3 \cup V_4}$ is bipartite.
\end{proof}

I will now prove a series of lemmas that culminate in theorem \ref{subquartic-acyclic}, the main theorem of this paper.

\begin{lemma} \label{bignegdeg} Let $\Sigma$ be a signed simple graph. If a vertex $v$ has $d^+(v) \leq 1$, then $\Sigma^v$ has no fully negative circles through $v$ and the set of fully negative circles of $\Sigma^v$ is a subset of the set of fully negative circles of $\Sigma$. Furthermore, if $\Sigma$ has a fully negative circle through $v$, then $\Sigma^v$ will have strictly fewer fully negative circles than $\Sigma$. \end{lemma}

\begin{proof}
Since $d^{-}(v) \leq 1$ in $\Sigma^v$, no fully negative circles contain $v$ in $\Sigma^v$. If $\Sigma^v$ has a fully negative circle that $\Sigma$ does not, it must be through $v$. But there are no fully negative circles through $v$ in $\Sigma^v$, so the set of fully negative circles of $\Sigma^v$ is a subset of the set of fully negative circles of $\Sigma$. 
\end{proof}

\begin{lemma}\label{bipartite-iff-delete-deg-3} Let $\Sigma$ be a signed graph. Let $X$ be the set of vertices in $\Sigma$ of degree 3 or less. The graph $\Sigma$ has a bipartite negation  set if and only if $\Sigma - X$ has a bipartite negation set. Furthermore, $\Sigma$ has an acyclic negation set if and only if $\Sigma - X$ has an acyclic negation set.  \end{lemma}

\begin{proof}
Suppose $\Sigma$ has a bipartite negation set $B$. Then $\Sigma - X$ has some negation set $C \subseteq B$. But since $C \subseteq B$, then $C$ is bipartite.

For the converse, suppose $\Sigma - X$ may be switched by $Y$ so the negative edge set is bipartite. That is, $E^-( (\Sigma - X)^Y )$ is bipartite.

There are two types of negative edges in $\Sigma^Y$: the negative edges of $(\Sigma - X)^Y$, and negative edges incident to vertices of $X$. 

In $\Sigma^Y$, switch only the vertices of $X$ that have negative degree at least 2 and suppose this gives the graph $\Sigma^{Y'}$. There are two types of negative edges in $\Sigma^{Y'}$: the negative edges of $(\Sigma-X)^Y$,  and edges which are incident to a vertex of negative degree 1 (i.e., a vertex in $X$). These two types of edges form a bipartition of $E^-(\Sigma^{Y'})$.

The negative edges incident to a vertex of negative degree 1 are not part of any fully negative circles. So if there is any fully negative circle of odd length in $\Sigma^{Y'}$, it must be contained entirely in $E^-((\Sigma - X)^Y)$. But $E^-((\Sigma - X)^Y)$ is bipartite, so any fully negative circles it contain must have even length. So $E^-(\Sigma^{Y'})$ is bipartite. 

In this proof, ``bipartite'' may be replaced with ``acyclic''.
\end{proof}

\begin{lemma}\label{4-core-lemma} A signed simple graph $\Sigma$ has a bipartite negation set if and only if the 4-core of $\Sigma$ has a bipartite negation set. Similarly, it has an acyclic negation set if and only if the 4-core of $\Sigma$ has an acyclic negation set.\end{lemma}

\begin{proof}
This follows immediately from lemma \ref{bipartite-iff-delete-deg-3}.
\end{proof}

Lemma \ref{4-core-lemma} gets us tantalizingly close to having a bipartite negation set for every signed simple graph with maximum degree 4. As we will see in theorem \ref{subquartic-acyclic}, this is almost true; out of all signed simple graphs with maximum degree 4 there is only one switching class that does not have a bipartite negation set. 

For the remainder of this section, I will be working with a signed simple quartic graph $\Sigma$. I may assume that for any vertex $v$, $d^{-}(v) \leq 2$ and thus each vertex is a member of at most one fully negative circle. Because I am concerned only with circles, I may assume the graph is 2-connected.

\begin{lemma}\label{nochord} If a signed simple quartic graph $\Sigma$ has a fully negative circle $C$ with a chord $xy$, then $\Sigma^{\lbrace x, y \rbrace}$ has strictly fewer fully negative circles. In particular, $C$ will not be fully negative in $\Sigma^{\lbrace x, y \rbrace}$.\end{lemma}

\begin{proof}
Suppose $\Sigma$ has a fully negative circle $C$ with a chord $e$ with ends $x$ and $y$. The chord $e$ cannot be negative because I assumed the vertices of $\Sigma$ had negative degree at most 2. So $e$ is positive. In $\Sigma^{\lbrace x, y \rbrace}$, $x$ and $y$ have negative degree 1 and thus $C$ is not fully negative. Since any fully negative circle introduced by this switching (i.e., a fully negative circle in $\Sigma^{\lbrace x, y \rbrace}$ but not in $\Sigma$) must pass through a subset of the switched vertices and the two switched vertices both have negative degree 1, no new fully negative circles were introduced in $\Sigma^{\lbrace x, y \rbrace}$. Thus $\Sigma^{\lbrace x, y \rbrace}$ has strictly fewer fully negative circles.
\end{proof}

\begin{figure}
\caption{The kind of circle we might expect to see in the proof of lemma \ref{BipartiteQuartic}.}
\label{cases-summary}
\begin{center}
\scalebox{0.65}{
\includegraphics[scale=1]{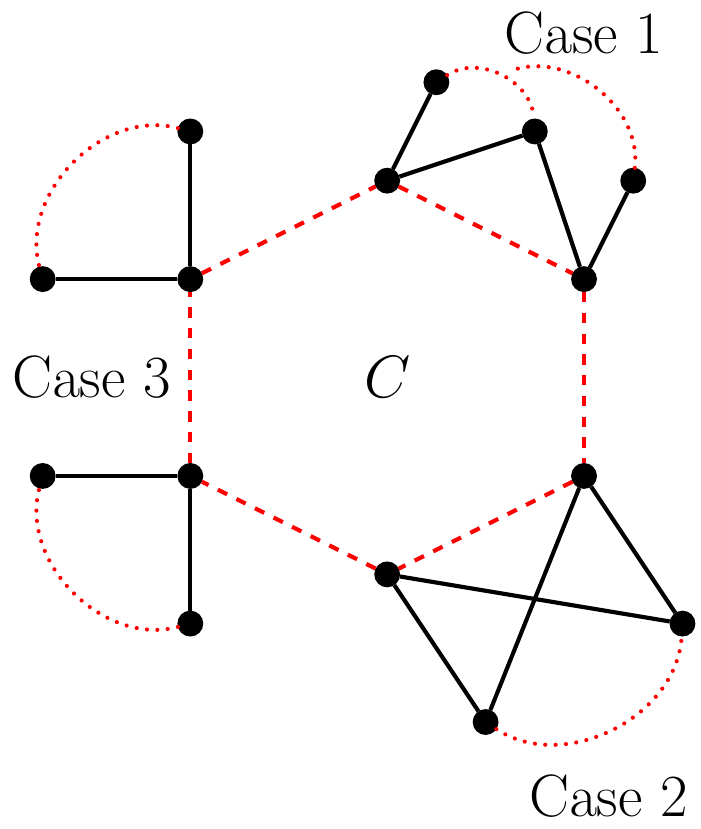}
}
\end{center}
\end{figure}

\begin{figure}
\caption{Case 1 in the proof of lemma \ref{BipartiteQuartic}.}
\label{case-1}
\begin{center}
\scalebox{0.65}{
\includegraphics[scale=1]{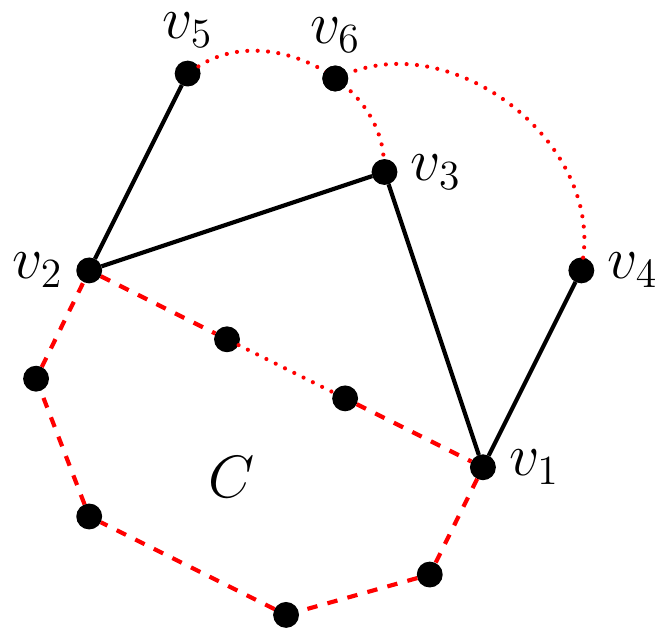}
}
\end{center}
\end{figure}

\begin{lemma} \label{BipartiteQuartic} Let $\Sigma$ be a signed simple quartic graph not in the switching class of $-K_5$. Then $\Sigma$ has an acyclic negation set. \end{lemma}

\begin{proof}
I show that $\Sigma$ may be switched to $\Sigma'$ with strictly fewer fully negative circles. Consider an arbitrary fully negative circle, $C$. Lemma \ref{nochord} allows us to assume $C$ has no chords. Lemma \ref{bignegdeg} allows us to assume each vertex in $C$ has negative degree exactly 2 and thus has exactly 2 positive neighbors.  These positive neighbors have negative degree 0, 1, 2, or 3. Suppose $v$ is a vertex in $C$ with positive neighbors $x$ and $y$. In $\Sigma^v$, $C$ will not be fully negative and there is possibly a fully negative circle through $x$, $v$, and $y$. If $x$ or $y$ has negative degree 0 in $\Sigma$, then it has negative degree 1 in $\Sigma^v$. Therefore, it is not part of any fully negative circle in $\Sigma^v$. If $x$ or $y$ has negative degree 2 or 3 in $\Sigma$ then it has negative degree 3 or 4 in $\Sigma^v$. So by lemma \ref{bignegdeg}, $\Sigma^{\lbrace v, x \rbrace }$ (or perhaps $\Sigma^{\lbrace v, y \rbrace}$) will have strictly fewer fully negative circles than $\Sigma$. So I only need to consider the case where the vertices positively adjacent to $C$ have negative degree 1 in $\Sigma$. I consider three cases determined by whether a pair of vertices in $C$ share both positive neighbors, share one positive neighbor, or no vertices of $C$ share any positive neighbors. 

I take a slight detour to make an important observation. Consider a vertex $v$ in $C$ with positive neighbors $x$ and $y$. In $\Sigma^v$, $C$ is not fully negative but there may be a fully negative circle through $v$, $x$, and $y$ that was not fully negative in $\Sigma$. For this to be true, there must be a fully negative path from $x$ to $y$ in $\Sigma$. If $x$ and $y$ are not joined by a fully negative path, then $\Sigma^v$ will have fewer fully negative circles than $\Sigma$. So I may assume that for each vertex in $C$, its positive neighbors are joined by a fully negative path. This gives us circles that look something like the circle in figure \ref{cases-summary}.

\textbf{Case 1:} Suppose two vertices, $v_1, v_2$, of $C$ share exactly 1 positive neighbor, say $v_3$, in common. Suppose the positive neighbors of $v_1$ are $v_3$ and $v_4$ and the positive neighbors of $v_2$ are $v_3$ and $v_5$. Suppose $v_3$ and $v_4$ are joined by the fully negative path $P_1$ and that $v_3$ and $v_5$ are joined by the fully negative path $P_2$. Note that $P_1$ and $P_2$ share some common segment containing $v_3$ or else $v_3$ would have negative degree larger than 1. Since $P_1$ and $P_2$ share a common segment, they must also share a common vertex with negative degree at least 3, say $v_6$. Refer to figure \ref{case-1}. In $\Sigma^{v_1}$, the circle $C$ is not fully negative, but the circle $v_1v_3P_1v_1$ is fully negative. In fact, there could be more negative circles in $\Sigma^{v_1}$ that are not fully negative in $\Sigma$, but they all must contain $v_3$ and $v_6$. But since $v_6$ has negative degree 3, then $\Sigma^{\lbrace v_1, v_6 \rbrace}$ has strictly fewer fully negative circles than $\Sigma$ by lemma \ref{bignegdeg}. 

\textbf{Case 2:} Suppose two vertices $v_1, v_2$ of $C$ share both of their positive neighbors, $v_3$ and $v_4$. If $v_3$ and $v_4$ are positively adjacent, I need only switch $\Sigma$ by $X = \lbrace v_1, v_2, v_3, v_4 \rbrace$. In $\Sigma^X$, $C$ is not fully negative, $v_1$ and $v_2$ each have negative degree at most 1 (depending on if $v_1$ and $v_2$ are adjacent in $C$), and $v_3$ and $v_4$ will have negative degree 0, thus there are no fully negative circles through any of $v_1, v_2, v_3$, or $v_4$. 

If $v_3$ and $v_4$ are negatively adjacent, I must separately consider whether $v_1$ and $v_2$ are adjacent or not in $C$. If $v_1$ and $v_2$ are not adjacent in $C$ then in $\Sigma^{v_1}$, $C$ is not fully negative but the triangle $v_1 v_3 v_4$ is fully negative. In this fully negative triangle, the vertex $v_3$ is positively adjacent to a vertex of negative degree 2. Thus in $\Sigma^{\lbrace v_1, v_3 \rbrace}$, $d^-(v_2) = 3$ so I may appeal to lemma \ref{bignegdeg} and conclude this subcase. 

Suppose $v_1$ and $v_2$ are adjacent in $C$. Let  $X = \lbrace v_1, v_2, v_4 \rbrace$. In $\Sigma^X$, $v_1$, $v_2$, and $v_3$ are positively adjacent to (respectively), $v_5$, $v_6$, and $v_7$ where  $\lbrace v_5, v_6, v_7 \rbrace \cap \lbrace v_1, v_2, v_3, v_4 \rbrace = \emptyset$. It is possible that $v_5 = v_6 = v_7$ in which case I must consider two possibilities: either $v_4 \sim v_5$ or $v_4 \not\sim v_5$. I am assuming $d^-(v_4) = d^-(v_5) = 1$ because if not, then I am done by previous remarks. If $v_4 \sim v_5$, then they are negatively adjacent and I have constructed a graph in the switching class of $-K_5$, the one exception to the theorem.  If $v_4 \not\sim v_5$ then I may switch $\Sigma^X$ by $Y = \lbrace v_1, v_2, v_3, v_4, v_5 \rbrace$. In $\Sigma^X$, the edges in $c(Y, V(\Sigma) - Y)$ are precisely the two negative edges incident with $v_4$ and $v_5$. So in $(\Sigma^X)^Y$, $d^+(v_4) = d^+(v_5) = 4$. Thus I may switch any vertex in the fully negative triangle $v_1v_2v_3$ and conclude this subcase. If $v_5, v_6$ and $v_7$ are not all equal, then two of the vertices in the fully negative triangle share exactly one positive neighbor in common, namely $v_4$, and thus I am done by case 1. 

\textbf{Case 3:} Assume all the vertices of $C$ have no positive neighbors in common. Recall that I have assumed that for each vertex in $C$, its positive neighbors are joined by a fully negative path. If I switch any vertex in $C$, I am simply shifting the discussion from $C$ to a new fully negative circle which I can assume also falls under case 3. By lemma \ref{nochord}, I may assume all of these circles are induced circles. Choose two of these induced circles, different than $C$, that each share a vertex with $C$, say $C_1$ and $C_2$. 

Since $\Sigma$ is 2-connected, I can choose two vertex-disjoint paths between $C_1$ and $C_2$. By appropriate choice of $C_1$ and $C_2$, I may assume one of these paths does not intersect $C$. Denote the vertex shared by $C$ and $C_1$ by $v_1$ and the vertex shared by $C$ and $C_2$ by $v_2$. The circles sharing a single vertex with $C$ can be assumed to have very special structure. Consider $\Sigma^{v_1}$, where $C$ is not fully negative but $C_1$ is fully negative.  If switching $v_1$ results in a net gain in number of negative circles, then $C_1$ must have vertices of negative degree at least 3 which I can switch to eliminate those circles. So I may always assume that when I switch the vertices on $C$, I eliminate exactly one negative circle (namely $C$) and create exactly one negative circle. If the negative circle I create does not have the structure I am examining here in case 3, then I can use what I know about cases 1 and 2 to achieve a switching with strictly fewer negative circles. 

So in $\Sigma^{\lbrace v_1, v_2 \rbrace }$, $C$ is not fully negative but $C_1$ and $C_2$ are fully negative. Suppose the path between $C_1$ and $C_2$ is $P = w_1w_2\ldots w_n$. The vertex $w_1$ has two positively adjacent neighbors which are joined by a fully negative path. If I switch $w_1$, I will have two fully negative circles joined by a path of length $n-1$. I proceed until vertices $w_1, \ldots, w_{n-2}$ have all been switched. Then I will have two fully negative circles joined by an edge. The edge from $w_{n-1}$ to $w_n$ must be positive because otherwise the negative circle $C_2$ would have a vertex of negative degree at least 3 in $\Sigma^{v_2}$ and I would be done by previous remarks. Since the edge joining the two fully negative circles must be positive, I can switch both its ends to reduce the number of fully negative circles. 
\end{proof}

Combining lemma \ref{4-core-lemma} and \ref{BipartiteQuartic}, I achieve the main result of this paper, theorem \ref{subquartic-acyclic}.

\begin{thm} \label{subquartic-acyclic}Let $\Sigma = (\Gamma, \sigma)$ be a simple signed graph with 4-core not in the switching class of $-K_5$. Then $\Sigma$ has an acyclic negation set.\end{thm} 

\begin{cor} \label{BipartiteSubquartic} A connected signed simple graph $\Sigma$ with maximum degree 4 and a vertex of degree less than 4 has an acyclic negation set. \end{cor}

The proof of lemma \ref{BipartiteQuartic} is constructive and gives us an algorithm for finding such a bipartite negative edge set. I give the algorithm below. Step 8 describes the first case of the proof, steps 9--13 describe case 2, and step 14 describes case 3.

\begin{enumerate}[Step 1.)]
\item Reduce the signed graph $\Sigma$ to the 4-core. If the 4-core has multiple components, we can execute the rest of the algorithm on each component separately.
\item \label{find-circle}Find a fully negative circle $C$. 
\item \label{after-finding-circle} If $C$ has a vertex with negative degree greater than 2, switch that vertex and go to step \ref{find-circle}. Otherwise, continue.
\item If $C$ has a chord, switch the ends of the chord and go to step \ref{find-circle}. Otherwise, continue.
\item  If any vertex $v$ in $C$ has its two positive neighbors not in the same component of the negative subgraph (i.e., the neighbors are not connected by a fully negative path), switch $v$ and go to step \ref{find-circle}. Otherwise, continue.
\item If any positive neighbor of the fully negative circle has negative degree other than 1, go to step \ref{step-3a}. Otherwise, go to step \ref{step-4}.
\item  \label{step-3a} If the negative degree of the positive neighbor is 0, switch the positive neighbor in $C$ and then go to step \ref{find-circle}. If the negative degree of the positive neighbor is greater than 1, switch that vertex and one of its neighbors in $C$. Then go to step \ref{find-circle}.
\item \label{step-4} (Case 1 of proof) If any two vertices in $C$, say $v_1$ and $v_2$, share exactly 1 positive neighbor, say $v_3$, look for a vertex, say $v_4$, of negative degree at least 3 on the fully negative path connecting the positive neighbors of $v_1$ or on the fully negative path connecting the positive neighbors of $v_2$. If $v_4$ is on the fully negative path connecting the positive neighbors of $v_1$ then switch $\lbrace v_1, v_4 \rbrace$ and go to step \ref{find-circle}. If $v_4$ is on the fully negative path connecting the positive neighbors of $v_2$ then switch $\lbrace v_2, v_4 \rbrace$ and go to step \ref{find-circle}. Otherwise, continue to step \ref{step-case-2}.
\item \label{step-case-2}(Case 2 of proof) If two vertices in $C$, say $v_1$ and $v_2$ share both of their positive neighbors (say $v_3$ and $v_4$), and those neighbors are \emph{not} negatively adjacent, then switch $\lbrace v_1, v_2, v_3, v_4 \rbrace$ and then go to step \ref{find-circle}. If their neighbors, $v_3$ and $v_4$ are negatively adjacent \emph{and} $v_1$ and $v_2$ are adjacent in $C$ then go to step \ref{5a}. If $v_3$ and $v_4$ are negatively adjacent \emph{and} $v_1$ and $v_2$ are not adjacent in $C$, go to step \ref{5b}. If no two vertices in $C$ share both of their positive neighbors, go to step \ref{step-6}.
\item \label{5a} Switch $\Sigma$ by $\lbrace v_1, v_2, v_4 \rbrace$. Vertices $v_1, v_2$ and $v_4$ have positive neighbors $v_5, v_6, v_7$ (distinct from $v_1, v_2, v_3, v_4$). If $v_5 = v_6 = v_7$ go to step \ref{5a1}. Otherwise, go to step \ref{5a2}.
\item \label{5a1} If $v_4$ is adjacent to $v_5$, then we have constructed a graph in the switching class of $-K_5$ and this graph does not have a bipartite negation set. If $v_4$ is not adjacent to $v_5$, then switch $\lbrace v_2, v_3, v_4, v_5 \rbrace$ and go to step \ref{find-circle}.  
\item \label{5a2} We have created a fully negative triangle, $v_1v_2v_3$ and two of them share exactly one positive neighbor in common, namely $v_4$. So set the triangle $v_1v_2v_3$ as $C$ and go to step \ref{step-4}.
\item \label{5b} If $v_1$ and $v_2$ are not adjacent in $C$, switch $\lbrace v_1, v_2, v_3 \rbrace$ (the two vertices in $C$ and one of their common neighbors). Go to step \ref{find-circle}.
\item \label{step-6} (Case 3 of Proof) (If this state of the algorithm is reached, $C$ should have no two vertices that share any of their positive neighbors.) If there are any fully negative circles that do not fall into this case, choose that circle as $C$ and go to step \ref{after-finding-circle}. Once all fully negative circles in the graph fall under this case, proceed. If no fully negative circles remain, go to step \ref{step-15}. Otherwise, choose two circles $C_1$ and $C_2$, each sharing a vertex with $C$, such that $C_1$ and $C_2$ are vertex disjoint and are joined by a path that does not intersect $C$. Suppose $C_1$ shares vertex $v_1$ with $C$ and $C_2$ shares vertex $v_2$ with $C$. Suppose the path joining $C_1$ to $C_2$ is $P = w_1w_2\ldots w_n$ where $w_1 \in C_1$ and $w_n \in C_2$. Switch $v_1$. Choose a fully negative circle through $v_1$  and $w_1$ as $C$ and go to step \ref{after-finding-circle}. If this step of the algorithm is reached again with this choice of $C$, switch the $w_i$ with largest index in $C$ and restart the algorithm with the negative circle containing $w_i$ and $w_{i+1}$ as $C$ (where $i+1 \leq n$). If, through repeated application of step \ref{step-6}, we achieve a fully negative circle that is adjacent to $w_n$ by a positive edge, then switch $\lbrace v_2, w_{n-1}, w_n \rbrace$.
\item \label{step-15} At this step, we have a switching of the 4-core with an acyclic negation set. Just as the 4-core is achieved by deleting vertices of degree 3 or less repeatedly, we will now add them back in stages and in the reverse order: the last vertices removed are the first added back. After we add back a group of vertices, we may switch them to guarantee that there are no fully negative circles through any vertex in that group. This is possible because of lemma \ref{bignegdeg}. Once the final group of vertices have been added on, we will have a switching of $\Sigma$ with an acyclic negative edge set.
\end{enumerate}

\section{Packing Negation Sets} 

For a negation set $B$, the \emph{packing number} of $B$ is the size of the largest family of disjoint negation sets to which $\lbrace B \rbrace$ can be enlarged. It is denoted by $\pi(B)$. In theorem \ref{packing-1}, I find an expression for $\pi(B)$. Theorem \ref{packing-algorithm} describes an algorithm for calculating that expression, and in its proof we see how we can recover such a family from the algorithm. 

\begin{thm}\label{packing-1} Let $\Sigma$ be an unbalanced signed graph such that $B:= E^-(\Sigma)$ is bipartite. Let $\mathcal{D}$ be the collection of all stable bipartitions $\lbrace X, Y \rbrace$ of $\Sigma{:}E^-$. Then $\pi(B) = \max_{\lbrace X, Y \rbrace \in \mathcal{D}} d^+(X, Y) + 1$.\end{thm}

\begin{proof}
Let $k := \max_{\lbrace X, Y \rbrace \in \mathcal{D}} d^+(X, Y)$. I begin by showing that $k+1$ is an upper bound of $\pi(B)$, and then I construct a family of size $k+1$. Let $\lbrace B_1, B_2 \rbrace$ be a stable bipartition of $\Sigma{:}E^-$ with $d^+(B_1, B_2) = k$. Let $P = v_0v_1v_2\cdots v_k$ be a fully positive path of length $k$ from $B_1$ to $B_2$. 

I claim that any negation set disjoint from $B$ must contain an edge of $P$. Suppose to the contrary that there exists a negation set $C$ disjoint from $B$ with no edge of $P$. Suppose $C = E^-(\Sigma^S)$ for some $S \subseteq V(\Sigma)$. To obtain a negative edge set disjoint from $B$, I must, in each component of $\Sigma{:}E^-$, switch all the vertices from either $B_1$ or $B_2$. By our original choice of the bipartition $\lbrace B_1, B_2 \rbrace$, I may assume that $B_1 \subseteq S$. I cannot switch any vertices of $B_2$ or else $C$ will not be disjoint from $B$. That is, $S \cap B_2 =\emptyset$. Since $B_1 \subseteq S$, then $v_0 \in S$. If $v_1 \notin S$, then the edge $v_0v_1 \in E(P)$ will be negative in $\Sigma^S$ and thus it will be an edge of $C$. Since this is not the case, then it must be that $v_1 \in S$. Similarly, if any $v_i \in S$ then $v_{i+1} \in S$, otherwise $C$ would contain an edge of $P$. However, this forces $v_k \in S \cap B_2$. But, this is impossible. So no such $C$ exists. Thus every negation set disjoint from $B$ must contain an edge of $P$. Therefore, $\pi(B) \leq k+1$.

To see this upper bound is always attainable, define $S_i = \lbrace v\in V(\Sigma) \mid d^+(B_1, v) \leq i\rbrace$. Define $N_i$ to be the negative edge set of $\Sigma^{S_i}$. I claim $\lbrace N_i: 0 \leq i \leq k-1 \rbrace \cup \lbrace B \rbrace$ forms a set of $k+1$ disjoint negation sets of $\Sigma$. Since each $N_i$ is the negative edge set of a switching of $\Sigma$, they are negation sets. Since $\Sigma^{S_i}$ has all of its negative edges in the cut $c(S_i, S_{i+1})$, and $c(S_i, S_{i+1}) \cap c(S_j, S_{j+1}) = \emptyset$ for $0 \leq i  < j \leq k-1$, then all of the $N_i$ are pairwise disjoint. Furthermore, since each $S_i$ contains $B_1$ and no vertex of $B_2$, then each $N_i$ is disjoint from $B$. So $\lbrace N_i: 0 \leq i \leq k-1 \rbrace \cup \lbrace B \rbrace$ is a family of $k+1$ disjoint negation sets. 
\end{proof}

Theorem \ref{packing-1} requires $\Sigma$ to be unbalanced because for a balanced graph $\max_{\lbrace X, Y \rbrace \in \mathcal{D}} d^+(X, Y)$ is either undefined (if $E^-(\Sigma) = \emptyset$ and therefore $\mathcal{D} = \emptyset$) or infinite (if $E^-(\Sigma)$ is a cut and therefore there are no positive paths from $X$ to $Y$). Moreover, the problem of packing negation sets in a balanced graph becomes equivalent to packing cuts, which is known to be a hard problem.

Though theorem \ref{packing-1} gives us a succinct expression for $\pi(B)$, if $B$ has $m$ components, it demands that we check $2^m$ positive distances. This can quickly get out of hand. Theorem $\ref{packing-algorithm}$ improves this by providing a fast algorithm. The bulk of the work in the algorithm is calculating the distances between all the vertices in the positive subgraph and checking some signed graphs for balance. 

Let $\Sigma$ be a signed graph such that $\Sigma{:}E^-$ is bipartite with components $C_1, \ldots, C_m$. For each $C_i$, let $\lbrace V_{i1}, V_{i2} \rbrace$ be its unique stable bipartition. Let $\mathcal{D}$ be the collection of all stable bipartitions of $\Sigma{:}E^-$. Let $\tau$ be the permutation $(1 2)$. Since $\lbrace V_{i\alpha}, V_{i\tau(\alpha)} \rbrace$ is a stable bipartition of $C_i$, a connected component of $\Sigma{:}E^-$, there are negative edges between $V_{i\alpha}$ and $V_{i\tau(\alpha)}$. Therefore, every stable bipartition of $\Sigma{:}E^-$ will have $V_{i\alpha}$ and $V_{i\tau(\alpha)}$ in opposite classes. 

Let $$E = \lbrace (V_{i\alpha}, V_{j\beta}) \mid i, j \in [m], \alpha, \beta \in [2], i\alpha \neq j\beta \rbrace$$ be a set of edges. Let $w_1 < w_2 < \cdots < w_l$ be the distinct real (i.e., exclude infinity) elements of $\lbrace d^+(V_{i\alpha}, V_{j\beta}) \mid i, j \in [m], \alpha, \beta \in [2], i\alpha \neq j\beta \rbrace $. Let $w_0 = 0$. For each $k \in [l]$ define: 
$$E_k' := \lbrace (V_{i\alpha}, V_{j\beta}) \in E \mid d^+(V_{i\alpha}, V_{j\beta}) \leq w_k \rbrace,$$
$$E_k'' := \lbrace (V_{i\tau(\alpha)}, V_{j\tau(\beta)}) \in E \mid d^+(V_{i\alpha}, V_{j\beta}) \leq w_k \rbrace,$$
$$E_k := E_k' \cup E_k''.$$

Define signed graphs $\Phi_1, \ldots, \Phi_l$ where $V(\Phi_k) = \lbrace V_{i\alpha} \mid i \in [m], \alpha \in [2] \rbrace$, $E^-(\Phi_k) = \lbrace (V_{i1}, V_{i2}) \mid i \in [m] \rbrace$, and  $E^+(\Phi_k) = E_k$. All of the signed graphs in $\lbrace \Phi_k \rbrace_{k=1}^l$ have the same vertex set and negative edge set. For $i < j$, $E^+(\Phi_i) \subseteq E^+(\Phi_j)$. 

Let $\mu \in \mathbb{Z}$ such that $w_\mu = \min \lbrace d^+(V_{i1}, V_{i2}) \mid i \in [m] \rbrace$. Then $\Phi_{\mu}$ has a negative digon. So not all of the signed graphs in $\lbrace \Phi_k \rbrace_{k=1}^l$ are balanced. Therefore, there exists a smallest integer $p$ such that the signed graphs in $\lbrace \Phi_{k} \rbrace_{k=1}^{p-1}$ are balanced and the signed graphs in $\lbrace \Phi_{k} \rbrace_{k=p}^l$ are unbalanced.

\begin{lemma}\label{lem-Bipart-Phi} Let $\Phi_0$ be a signed graph with vertex set $\lbrace V_{i1}, V_{i2} \rbrace_{i=1}^m$, and negative edge set $E^-(\Phi_0) = \lbrace (V_{i1}, V_{i2}) \rbrace_{i=1}^m$. Let $S$ be a set of positive edges on $V(\Phi_0)$. Let $\Phi:= \Phi_0 \cup S$. Then $\Phi$ is balanced if and only if there exists a stable bipartition $\lbrace B_1, B_2 \rbrace$ of $\Sigma{:} E^-$ such that for all $V_{i\alpha}, V_{j\beta}$ in $V(\Phi)$:
\begin{enumerate}
	\item\label{bi-con-1} $V_{i\alpha}$ and $V_{j\beta}$ are in the same class of $\lbrace B_1, B_2 \rbrace$ if there exists a positive path from $V_{i\alpha}$ to $V_{j\beta}$ in $\Phi$, 
	\item \label{bi-con-2} $V_{i\alpha}$ and $V_{j\beta}$ are in opposite classes of $\lbrace B_1, B_2 \rbrace$ if there exists a negative path from $V_{i\alpha}$ to $V_{j\beta}$ in $\Phi$.
\end{enumerate} 
\end{lemma}
	
\begin{proof}
Let $\phi$ be the sign function of $\Phi$. It suffices to consider the case where $\Phi$ is connected. 

Suppose $\Phi$ is balanced. Then $\Phi$ satisfies Harary's path condition for balance: for any two vertices in $\Phi$ every path connecting them has the same sign. I will construct a bipartition of $V(\Sigma{:} E^-)$ satisfying conditions \ref{bi-con-1} and \ref{bi-con-2}. Pick any vertex $X \in \Phi$. Let $P$ be the set of vertices in $\Phi$ that are connected to $X$ by a positive path in $\Phi$ and $X$ itself. Let $N$ be the set of vertices in $\Phi$ that are connected to $X$ by a negative path in $\Phi$. Because of Harary's path condition, $P \cap N = \emptyset.$ Since I assumed $\Phi$ is connected, $P \cup N = V(\Phi)$. Thus $\lbrace P, N \rbrace$ is a bipartition of $V(\Phi)$. Let $B_1 = \bigcup_{V \in P} V$ and $B_2 = \bigcup_{V \in N} V$. 

I claim that $\lbrace B_1, B_2 \rbrace$ is a stable bipartition of $\Sigma{:} E^-$. The set $V(\Phi) = \lbrace V_{i\alpha} \mid i \in [m], \alpha \in [2] \rbrace$ is a partition of $V(\Sigma{:} E^-)$. Since $\lbrace P, N \rbrace$ is a bipartition of $V(\Phi)$, then $\lbrace B_1, B_2 \rbrace$ is just a coursening of the partition $V(\Phi)$ of $V(\Sigma{:} E^-)$. So $\lbrace B_1, B_2 \rbrace$ is a bipartition of $V(\Sigma{:}E^-)$. In particular, this means $B_1 \cap B_2 = \emptyset$.

To see that the bipartition is stable, consider two vertices $x, y \in B_1$ and suppose by way of contradiction that $x \sim y$ in $\Sigma{:}E^-$. Then $x$ and $y$ are in the same component of $\Sigma{:}E^-$. In particular, since they are adjacent in $\Sigma{:}E^-$, they must be in opposite classes of the stable bipartition of that component. That is, there must exist some $V_{i1}$ and $V_{i2}$ in $V(\Phi)$ with $x \in V_{i1}$ and $y \in V_{i2}$. Without loss of generality, suppose $V_{i1} \in P$. So there exists a positive path from $X$ to $V_{i1}$ in $\Phi$. In $\Phi$, $V_{i1}$ is negatively adjacent to $V_{i2}$ so there exists a negative path from $X$ to $V_{i2}$ in $\Phi$. So $V_{i2} \in N$. Thus $V_{i2} \subseteq B_2$ and hence $y \in B_2$. But this is a contradiction because I assumed $y \in B_1$ and have already established that $B_1 \cap B_2 = \emptyset$. Therefore $\lbrace B_1, B_2 \rbrace$ is a stable bipartition of $V(\Phi)$. 

I will now show that $\lbrace B_1, B_2 \rbrace$ satisfies conditions \ref{bi-con-1} and \ref{bi-con-2}. By the construction of $\lbrace B_1, B_2 \rbrace$, $V_{i\alpha}$ is in the same class as $X$ if they are joined by a positive path and they are in opposite classes if they are joined by a negative path. Suppose $V_{i\alpha}$ and $V_{j\beta}$ are connected by a positive path in $\Phi$. Let $P_1$ be a path in $\Phi$ from $V_{i\alpha}$ to $X$ and $P_2$ be a path in $\Phi$ from $X$ to $V_{j\beta}$. Then the concatenation of these paths, $P_1P_2$, is a path from $V_{i\alpha}$ to $V_{j\beta}$.

Because of Harary's path condition for balance, $\phi(P_1P_2) = +$. Since $\phi(P_1P_2) = \phi(P_1)\phi(P_2)$, then either $\phi(P_1)$ and $\phi(P_2)$ are both positive or both negative. If $\phi(P_1)$ and $\phi(P_2)$ are both positive then $V_{i\alpha}$ is connected to $X$ by a positive path. Therefore, $V_{i\alpha}$ and $X$ are in the same class of $\lbrace B_1, B_2 \rbrace$. Similarly, $V_{j\beta}$ and $X$ are connected by a positive path. Therefore, they are in the same class of $\lbrace B_1, B_2 \rbrace$. Since $V_{i\alpha}$ and $V_{j\beta}$ are both in the same class as $X$, then $V_{i\alpha}$ and $V_{j\beta}$ are in the same class. If $\phi(P_1)$ and $\phi(P_2)$ are both negative, then both $V_{i\alpha}$ and $V_{j\beta}$ are in the class not containing $X$, i.e.,  $V_{i\alpha}$ and $V_{j\beta}$ are in the same class. This shows that condition \ref{bi-con-1} holds. 

Showing that condition \ref{bi-con-2} holds is similar. Suppose $V_{i\alpha}$ and $V_{j\beta}$ are joined by a negative path in $\Phi$. Let $P_1$ be a path from $V_{i\alpha}$ to $X$ and $P_2$ be a path from $X$ to $V_{j\beta}$. Then the concatenation of these paths, $P_1P_2$, is a path from $V_{i\alpha}$ to $V_{j\beta}$. 

Because of Harary's path condition for balance, $\phi(P_1P_2) = -$. Since $\phi(P_1P_2) = \phi(P_1)\phi(P_2)$, then $\phi(P_1)$ and $\phi(P_2)$ have opposite signs. Suppose that $\phi(P_1) = +$ and $\phi(P_2) = -$. Then $V_{i\alpha}$ is connected to $X$ by a positive path and therefore they are in the same class of $\lbrace B_1, B_2 \rbrace$. Similarly, $V_{j\beta}$ and $X$ are connected by a negative path and therefore they are in opposite classes of $\lbrace B_1, B_2 \rbrace$. Thus $V_{i\alpha}$ and $V_{j\beta}$ are in opposite classes of $\lbrace B_1, B_2 \rbrace$. So condition \ref{bi-con-2} holds.

Conversely, if $\Phi$ is unbalanced, I will show that there cannot be a stable bipartition satisfying \ref{bi-con-1} and \ref{bi-con-2}. By way of contradiction, suppose there exists such a stable bipartition. Since $\Phi$ is unbalanced, there exists $V_{i\alpha}$ and $V_{j\beta}$ connected by both a positive and negative path in $\Phi$. Since $V_{i\alpha}$ and $V_{j\beta}$ are connected by a positive path, then $V_{i\alpha}$ and $V_{j\beta}$ are in the same class of the stable bipartition of $\Sigma{:}E^-$. Since $V_{i\alpha}$ and $V_{j\beta}$ are connected by a negative path, then they are in opposite classes of the stable bipartition of $\Sigma{:}E^-$. This is a contradiction. So if $\Phi$ is unbalanced, there is no stable bipartition of $\Sigma{:}E^-$ satisfying \ref{bi-con-1} and \ref{bi-con-2}.
\end{proof}

\begin{thm}\label{packing-algorithm}
Let $p$ be the smallest integer such that $\Phi_p$ is unbalanced. Then $\max_{\lbrace X, Y \rbrace \in \mathcal{D}} d^+(X, Y) = w_p$.
\end{thm}

\begin{proof}
Let $q \in [m]$. Suppose $\Phi_q$ is balanced. Then by lemma \ref{lem-Bipart-Phi}, there exists a stable bipartition  $\lbrace Q_1, Q_2 \rbrace$ of $\Sigma{:}E^-$ satisfying lemma \ref{lem-Bipart-Phi}(\ref{bi-con-1}) and \ref{lem-Bipart-Phi}(\ref{bi-con-2}). I will show that $d^+(Q_1, Q_2) > w_q$. By way of contradiction, suppose there exists $x \in V_{r\gamma} \subseteq Q_1$ and $y \in V_{s\theta} \subseteq Q_2$ such that $d^+(x, y) \leq w_q$. Then $d^+(V_{r\gamma}, V_{s\theta}) \leq w_q$. So $(V_{r\gamma}, V_{s\theta}) \in E^+(\Phi_q)$. So $V_{r\gamma}$ is positively adjacent to $V_{s\theta}$ in $\Phi_q$. Thus there is a positive path from $V_{r\gamma}$ to $V_{s\theta}$ in $\Phi_q$. So by lemma \ref{lem-Bipart-Phi},  $V_{r\gamma}$ and $V_{s\theta}$ are in the same class of $\lbrace Q_1, Q_2 \rbrace$. So $x$ and $y$ are in the same class of $\lbrace Q_1, Q_2 \rbrace$. This is a contradiction because I assumed that $x$ and $y$ were in different classes. Thus $d^+(Q_1, Q_2) > w_q$.  

Next I prove that $ d^+(Q_1, Q_2) \in \lbrace w_k \rbrace_{i=1}^l$. Suppose $x \in Q_1$ and $y \in Q_2$ such that $d^+(Q_1, Q_2) = d^+(x, y)$. For some $i\alpha$ and $j\beta$,  $x \in V_{i\alpha}$ and $y \in V_{j\beta}$. So $V_{i\alpha} \subseteq Q_1$ and $V_{j\beta} \subseteq Q_2$. Clearly, $d^+(x, y) \geq d^+(V_{i\alpha}, V_{j\beta})$. But if $ d^+(V_{i\alpha}, V_{j\beta}) < d^+(x, y)$ then $d^+(Q_1, Q_2) < d^+(x, y)$ because $d^+(Q_1, Q_2)$ is bound above by $d^+(V_{i\alpha}, V_{j\beta})$. This is a contradiction, so $d^+(x, y) = d^+(V_{i\alpha}, V_{j\beta}) = d^+(Q_1, Q_2)$. So $d^+(Q_1, Q_2) \in \lbrace w_k \rbrace_{k=1}^l$. 

Since $\Phi_{p-1}$ is balanced, there exists a stable bipartition $\lbrace B_1, B_2 \rbrace$ of $\Sigma{:}E^-$ such that $d^+(B_1, B_2) > w_{p-1}$. Thus $d^+(B_1, B_2) \geq w_p$.

Finally, I show that there is no stable bipartition of $\Sigma{:}E^-$ with distance between its classes larger than $w_p$. By way of contradiction, suppose there exists a stable bipartition $\lbrace B_3, B_4 \rbrace$ of $\Sigma{:}E^-$ with $d^+(B_3, B_4) > w_p$. Recall that any stable bipartition of $\Sigma{:}E^-$ has $V_{i1}$ and $V_{i2}$ in opposite classes. 

Since $d^+(B_3, B_4) > w_p$, any two vertices $V_{r\gamma}, V_{s\theta}$ with $d^+(V_{r\gamma}, V_{s\theta}) \leq w_p$ must be contained in the same class of $\lbrace B_3, B_4 \rbrace$. To see this, suppose to the contrary that $V_{r\gamma}$ and $V_{s\theta}$ are in opposite classes. Suppose $V_{r\gamma} \subseteq B_3$, and $V_{s\theta} \subseteq B_4$. Let $x \in V_{r\gamma}$ and $y \in V_{s\theta}$ having $d^+(x, y) = d^+(V_{r\gamma}, V_{s\theta})$. Then since $x \in B_3$ and $y \in B_4$, $d^+(B_3, B_4) \leq d^+(x, y) = d^+(V_{r\gamma}, V_{s\theta}) \leq w_p$. This is a contradiction because I assumed $d^+(B_3, B_4) > w_p$. So any $V_{r\gamma}$ and $V_{s\theta}$ with $d^+(V_{r\gamma}, V_{s\theta}) \leq w_p$ must be in the same class of $\lbrace B_3, B_4 \rbrace$. 

I will show that the structure of $\Phi_p$ is connected to the structure of $\lbrace B_3, B_4 \rbrace$. I will then use lemma \ref{lem-Bipart-Phi} to reach a contradiction and conclude the proof. The signed graph $\Phi_p$ has positive edge set \begin{align*}E^+(\Phi_p) &= \lbrace (V_{i\alpha}, V_{j\beta}) \mid d^+(V_{i\alpha}, V_{j\beta}) \leq w_p \rbrace \cup \lbrace (V_{i\tau(\alpha)}, V_{j\tau(\beta)} ) \mid d^+(V_{i\alpha}, V_{j\beta}) \leq w_p \rbrace\\ &= E_p' \cup E_p'' \end{align*} and negative edge set $$E^-(\Phi_p) = \lbrace (V_{i1}, V_{i2}) \mid i \in [m] \rbrace.$$

For $(V_{i\alpha}, V_{j\beta}) \in E^+(\Phi_p)$, either $(V_{i\alpha}, V_{j\beta}) \in E_p'$ or $(V_{i\alpha}, V_{j\beta}) \in E_p''$. If $(V_{i\alpha}, V_{j\beta}) \in E_p'$ then $d^+(V_{i\alpha}, V_{j\beta}) \leq w_p$. So $V_{i\alpha}, V_{j\beta}$ are in the same class of $\lbrace B_3, B_4 \rbrace$ because of our previous discussion. If $(V_{i\alpha}, V_{j\beta}) \in E_p''$ then $(V_{i\tau(\alpha)}, V_{j\tau(\beta)}) \in E_p'$ and hence $d^+(V_{i\tau(\alpha)}, V_{j\tau(\beta)}) \leq w_p$. So $V_{i\tau(\alpha)}$ and $V_{j\tau(\beta)}$ are in the same class of $\lbrace B_3, B_4 \rbrace$ and therefore $V_{i\alpha}$ and $V_{j\beta}$ are in the same class of $\lbrace B_3, B_4 \rbrace$.

For $(V_{i\alpha}, V_{j\beta}) \in E^-(\Phi_p)$, $(V_{i\alpha}, V_{j\beta}) = (V_{i1}, V_{i2})$ for some $i$. Since $V_{i1}$ and $V_{i2}$ must be in opposite classes of $\lbrace B_3, B_4 \rbrace$ then $V_{i\alpha}$ and $V_{j\beta}$ are in opposite classes of $\lbrace B_3, B_4 \rbrace$. 

In summary: if $(V_{i\alpha}, V_{j\beta}) \in E^+(\Phi_p)$, then $V_{i\alpha}$ and $V_{j\beta}$ are in the same class of $\lbrace B_3, B_4 \rbrace$; if $(V_{i\alpha}, V_{j\beta}) \in E^-(\Phi_p)$ then $V_{i\alpha}$ and $V_{j\beta}$ are in opposite classes of $\lbrace B_3, B_4 \rbrace.$ By extension, if $V_{i\alpha}$ and $V_{j\beta}$ are joined by a positive path in $\Phi_p$ then $V_{i\alpha}$ and $V_{j\beta}$ are in the same class of $\lbrace B_3, B_4 \rbrace$. Similarly, if $V_{i\alpha}$ and $V_{j\beta}$ are joined by a negative path in $\Phi_p$, then $V_{i\alpha}$ and $V_{j\beta}$ are in opposite classes of $\lbrace B_3, B_4 \rbrace.$ So $\lbrace B_3, B_4 \rbrace$ is a stable bipartition of $\Sigma{:}E^-$ satisfying 1 and 2 of lemma \ref{lem-Bipart-Phi} with corresponding signed graph $\Phi_p$. So by lemma \ref{lem-Bipart-Phi}, $\Phi_p$ should be balanced. But by hypothesis, $\Phi_p$ is unbalanced. So no stable bipartition $\lbrace B_3, B_4 \rbrace$ of $\Sigma{:}E^-$ with $d^+(B_3, B_4) > w_p$ exists. So $\max_{ \lbrace X, Y \rbrace \in \mathcal{D}} d^+(X, Y) = w_p$. 
\end{proof}

\bibliographystyle{amsplain}
\bibliography{lacasse-negation-sets-bib}

\end{document}